\documentclass[10pt]{article}

\usepackage[a4paper,margin=0.7in]{geometry}
\usepackage{amsmath,amssymb,amsthm}
\usepackage{microtype}

\newtheorem{proposition}{Proposition}
\newtheorem{theorem}{Theorem}
\newtheorem{lemma}[theorem]{Lemma}

\newcommand{\SDD}{\operatorname{SDD}}

\title{Corrigendum to ``On the Maximum Symmetric Division Deg Index of
$k$-Cyclic Graphs [J. Math. 2022, (2022) \#7783128]''}
\author{Abeer M. Albalahi and Akbar Ali\footnote{Corresponding author: akbarali.maths@gmail.com}\\[2mm]
Department of Mathematics, Faculty of Science,\\ University of Ha\!'il, Ha\!'il, Saudi Arabia}
\date{}

\begin{document}

\maketitle

\begin{abstract}
The proof of Proposition 1 in the paper [J. Math. 2022, (2022) \#7783128] contains a gap, but its statement is correct. Here, we fix this gap.

\end{abstract}

\section{Complete Proof of Proposition~1}

 The degree
of a vertex $u$ in a graph $G$ is denoted by $d_u(G)$, or simply by $d_u$ when there is no
danger of confusion.
The set $N_G(u)\cup \{u\}$ is denoted by $N_G[u]$, where $N_G(u)= \{x\in V(G):xu\in E(G)\}$. The symmetric division
deg index of a graph $G$ is defined as
$\SDD(G)=\sum_{uv\in E(G)}\Psi\bigl(d_u,d_v\big),$
where $\Psi(x,y)=\frac{x^2+y^2}{xy}$.
A connected graph with $n$ vertices is said to be a
$k$-cyclic graph if it has $n-1+k$ edges for some non-negative integer $k$. We follow \cite{AlbalahiAli2022} for other notation and terminology.

In the proof of Proposition~1 of \cite{AlbalahiAli2022} (see Equation (3) there), it was claimed that every term in the following summation is positive by Lemma 1 of \cite{AlbalahiAli2022}:
\begin{align*}
&\sum_{i=1}^r\big[\Psi\bigl(d_v+r,d_{v_i'}\big)
-\Psi\bigl(d_{v'},d_{v_i'}\big)\big]=\sum_{i=1}^r
\frac{\bigl(d_v+r-d_{v'}\big)
\bigl(d_{v'}[d_v+r]-[d_{v_i'}]^2\big)}
{d_{v_i'}d_{v'}[d_v+r]}.
\end{align*}
However, Lemma 1 of \cite{AlbalahiAli2022} requires $d_{v'}\ge d_{v'_i}$, which may or may not hold in the proposition under discussion.  Therefore, the claimed positivity
does not follow. To fix this gap, we use the following lemma derived from the proof of Proposition~1 in \cite{AlbalahiAli2022}.

\begin{lemma}\label{lem:transfer}
Let $v\in V(G)$ be a vertex of maximum degree. Let $vw\in E(G)$ and let
\[
\{w_1,\ldots,w_s\}
=\{x\in N_G(w)\setminus N_G[v]:d_x\le d_w\}\ne\emptyset.
\]
If $G'$ is the graph formed from $G$ by removing the edges $ww_1,\ldots,ww_s$ and adding the edges $vw_1,\ldots,vw_s$,
then $\SDD(G)-\SDD(G')<0$.
\end{lemma}

\begin{proposition}\cite{AlbalahiAli2022}\label{thm:main}
If a graph $G$ possesses the largest SDD index
among all $k$-cyclic graphs of order $n$, then the maximum degree of $G$ is $n-1$, where $0\le k\le n-2$.
\end{proposition}

\begin{proof}
For $n\le 3$, there is nothing to prove. Hence, we assume that $n\ge 4$. Let $\Delta$ be the maximum degree of $G$. We suppose, to the contrary, that $\Delta\le n-2$,
and choose a vertex $v\in V(G)$ of degree $\Delta$. \\[2mm]
{\bf Claim 1.}    The
subgraph induced by $V(G)\setminus N_G[v]$ has no edges.

\begin{proof}[Proof of Claim 1.]
    Suppose that a component $Q$ of this subgraph contains an edge. Choose
$z\in V(Q)$ such that $d_z(G)=\max\{d_q(G):q\in V(Q)\}$. Then, $N_G(z)\setminus N_G[v]$ is nonempty. We note that
\begin{equation}\label{eq:Q-maximum}
d_x(G)\le d_z(G)
\quad\text{for every }x\in N_G(z)\setminus N_G[v],
\end{equation}
because every
$x\in N_G(z)\setminus N_G[v]$ belongs to $Q$.\\[2mm]
{\bf Case 1.} In $G$, the vertex $z$ has a neighbor in $N_G(v)$.\\
Let
$N_G(z)\setminus N_G[v]=\{x_1,\ldots,x_s\}$, where $s\ge1$. Form
$G'$ by removing the edges $zx_1,\dots,zx_s$ and adding the edges $vx_1,\dots,vx_s$. We note that
$d_z(G)-s\ge1$ because the set $N_G(z)\cap N_G(v)$ is nonempty, by the
assumption of the present case.
Therefore,
\begin{align}\label{eq-001}
\SDD(G)-\SDD(G')
={}&\sum_{u\in N_G(v)\setminus N_G(z)}
 \bigl[\Psi(\Delta,d_u(G))-\Psi(\Delta+s,d_u(G))\big]\nonumber\\
&+\sum_{y\in N_G(v)\cap N_G(z)}
 \bigl[\Psi(\Delta,d_y(G))+\Psi(d_z(G),d_y(G))\big]\nonumber\\
&-\sum_{y\in N_G(v)\cap N_G(z)}
 \bigl[\Psi(\Delta+s,d_y(G))
       +\Psi(d_z(G)-s,d_y(G))\big]\nonumber\\
&+\sum_{i=1}^{s}
 \bigl[\Psi(d_z(G),d_{x_i}(G))
       -\Psi(\Delta+s,d_{x_i}(G))\big].
\end{align}
Each term in the first and last sums in \eqref{eq-001} is negative by Lemma 1 of \cite{AlbalahiAli2022}. Also, we note that
\begin{align*}
&\Psi(\Delta,d_y(G))+\Psi(d_z(G),d_y(G))-\Psi(\Delta+s,d_y(G))
       -\Psi(d_z(G)-s,d_y(G))\\
&= -\frac{s d_y(G)(\Delta+d_z(G))(\Delta-d_z(G)+s)}
 {\Delta d_z(G)(\Delta+s)(d_z(G)-s)}<0.
\end{align*}
Therefore,
\eqref{eq-001} gives
$\SDD(G)-\SDD(G')<0$, contrary to the choice of $G$.\\[2mm]
{\bf Case 2.}
In $G$, the vertex $z$ has no neighbor in $N_G(v)$.\\
Let $P$ be a shortest path from $z$ to $N_G(v)$, and let $y$ be the
neighbor of $z$ on $P$. Since $z$ has no neighbor in $N_G(v)$, the path
$P$ has length at least $2$. Therefore, $y\notin N_G[v]$, and hence, $z,y\in V(Q)$. Moreover, $2\le d_y(G)\le d_z(G)$. Let
$N_G(z)\setminus\{y\}=\{x_1,\ldots,x_{d_z(G)-1}\}$. We form $G'$ by
removing all the edges $zx_1,\ldots,zx_{d_z(G)-1}$ and adding all the edges $vx_1,\ldots,vx_{d_z(G)-1}$.
Since every $x_i$ belongs to $Q$, from
\eqref{eq:Q-maximum} it follows that $d_{x_i}(G)\le d_z(G)$. Here, we have
\begin{align}\label{eq-002}
\SDD(G)-\SDD(G')
={}&\sum_{u\in N_G(v)}
 \bigl[\Psi(\Delta,d_u(G))
       -\Psi(\Delta+d_z(G)-1,d_u(G))\big]\nonumber\\
&+\sum_{i=1}^{d_z(G)-1}
 \bigl[\Psi(d_z(G),d_{x_i}(G))
       -\Psi(\Delta+d_z(G)-1,d_{x_i}(G))\big]\nonumber\\
&+\Psi(d_z(G),d_y(G))-\Psi(1,d_y(G)).
\end{align}
Each term in the first sum of \eqref{eq-002} is negative by Lemma 1 of \cite{AlbalahiAli2022}. Also, for every $i\in\{1,\dots,d_z(G)-1\}$, we note that $\Psi(d_z(G),d_{x_i}(G))
       -\Psi(\Delta+d_z(G)-1,d_{x_i}(G))$ strictly increases in $d_{x_i}(G)$, and hence,
       $$\Psi(d_z(G),d_{x_i}(G))
       -\Psi(\Delta+d_z(G)-1,d_{x_i}(G)) \le \Psi(d_z(G),d_{z}(G))
       -\Psi(\Delta+d_z(G)-1,d_{z}(G)).$$
       Moreover, by Lemma 1 of \cite{AlbalahiAli2022}, we have $-\Psi(\Delta+d_z(G)-1,d_{z}(G))\le-\Psi(2d_z(G)-1,d_{z}(G))$.
        Finally, we note that $\Psi(d_z(G),d_y(G))-\Psi(1,d_y(G))$ strictly decreases in $d_y(G)$. Hence, \eqref{eq-002} yields
\begin{align*}
\SDD(G)-\SDD(G')
<{}&(d_z(G)-1)
 \bigl[\Psi(d_z(G),d_{z}(G))
-\Psi(2d_z(G)-1,d_{z}(G))\big]\\
&+\Psi(d_z(G),2)-\Psi(1,2)=-\frac{(d_z(G)-1)(5d_z(G)-2)}{2d_z(G)(2d_z(G)-1)}<0,
\end{align*}
which is again a contradiction.

 In both cases, we arrive at a contradiction. Therefore, $V(G)\setminus N_G[v]$ is an independent set. This completes the proof of Claim 1.
\end{proof}

Now, let $x\notin N_G[v]$. Since $G$ is connected and
$V(G)\setminus N_G[v]$ is independent (by Claim 1), the vertex $x$ has a neighbor
$w\in N_G(v)$.  If
$d_x(G)\le d_w(G)$, then the set occurring in
Lemma~\ref{lem:transfer} is nonempty because it contains $x$, and hence, by Lemma~\ref{lem:transfer} there exists a $k$-cyclic graph of order $n$ such that $\SDD(G)-\SDD(G')<0$, which contradicts the extremal choice of
$G$. Hence, in what follows, we assume that
\begin{equation}\label{eq:degree-comparison}
d_x(G)>d_w(G)
\quad\text{whenever }w\in N_G(v),\ x\notin N_G[v],\text{ and }wx\in E(G).
\end{equation}
Since $w$ is a common neighbor of $v$ and $x$, we have
$d_w(G)\ge2$. Thus, (4) gives
$d_x(G)\ge d_w(G)+1\ge3$, and hence
\begin{equation}\label{eq:outside-degree}
d_x(G)\ge3\quad\text{for every }x\notin N_G[v].
\end{equation}
Also,
\begin{equation}\label{eq-new-22}
r:=n-\Delta-1=|V(G)\setminus N_G[v]|.
\end{equation}
Since $\Delta\le n-2$, we have $r\ge1$. Since $G$ is connected, there exists a vertex
$w\in N_G(v)$ having a neighbor outside $N_G[v]$. We fix such a vertex
$w$ and let
$
N_G(w)\setminus N_G[v]=\{x_1,\ldots,x_q\}$
and $p=1+|N_G(w)\cap N_G(v)|.
$
Certainly, $q\ge1$. We note that the set $N_G(w)$ is partitioned into the vertex
$\{v\}$, the subset  of $N_G(w)\cap N_G(v)$ consisting of $p-1$ elements, and
$\{x_1,\ldots,x_q\}$. Therefore,
$d_w(G)=1+(p-1)+q=p+q$. Hence,
$\Delta-p\ge q$. Also,
the assumption $k\le n-2$ gives
\begin{equation}\label{eq-new-1}
|E(G)|=n-1+k\le2n-3.
\end{equation}
Form $G'$ by removing the edges $wx_1,\ldots,wx_q$ and adding the
edges $vx_1,\ldots,vx_q$.
Hence, we have
\begin{align}\label{eq-new-005}
\SDD(G)-\SDD(G')
={}&\sum_{z\in N_G(v)\setminus N_G[w]}
 \bigl[\Psi(\Delta,d_z(G))-\Psi(\Delta+q,d_z(G))\big]\nonumber\\
&+\sum_{z\in N_G(v)\cap N_G(w)}
 \bigl[\Psi(\Delta,d_z(G))+\Psi(d_w(G),d_z(G))\big]\nonumber\\
&-\sum_{z\in N_G(v)\cap N_G(w)}
 \bigl[\Psi(\Delta+q,d_z(G))+\Psi(p,d_z(G))\big]\nonumber\\
&+\Psi(\Delta,d_w(G))-\Psi(\Delta+q,p)+\sum_{i=1}^{q}
 \bigl[\Psi(d_w(G),d_{x_i}(G))
       -\Psi(\Delta+q,d_{x_i}(G))\big].
\end{align}
Since $d_w(G)=p+q$, we have
\begin{align*}
&\Psi(\Delta,d_z(G))+\Psi(d_w(G),d_z(G))-\Psi(\Delta+q,d_z(G))-\Psi(p,d_z(G)) \\
&=-\frac{q d_z(G)(\Delta+p+q)(\Delta-p)}
 {p\Delta (p+q)(\Delta+q)}<0,
\end{align*}
and hence, from \eqref{eq-new-005}, it follows that
\begin{align}\label{eq:main-difference}
\SDD(G)-\SDD(G')
\le{}&\sum_{z\in N_G(v)\setminus N_G[w]}
\bigl[\Psi(\Delta,d_z(G))-\Psi(\Delta+q,d_z(G))\big]\nonumber\\
&+\Psi(\Delta,d_w(G))-\Psi(\Delta+q,p)+\sum_{i=1}^{q}
 \bigl[\Psi(d_w(G),d_{x_i}(G))
       -\Psi(\Delta+q,d_{x_i}(G))\big]\nonumber\\
={}&-q\sum_{z\in N_G(v)\setminus N_G[w]}
 \left(\frac1{d_z(G)}-\frac{d_z(G)}{\Delta(\Delta+q)}\right)\nonumber\\
&-\frac{q(\Delta+p+q)^2(\Delta-p)}
 {p\Delta (p+q)(\Delta+q)}-(\Delta-p)\sum_{i=1}^{q}
 \left(\frac1{d_{x_i}(G)}-
 \frac{d_{x_i}(G)}{(p+q)(\Delta+q)}\right).
\end{align}

Let $h$ be
the number of edges whose two end-vertices lie in $N_G(v)$, and let $e$ be
the number of those edges whose one end-vertex belongs to $N_G(v)$ and the other one belongs to $V(G)\setminus N_G[v]$. Since $V(G)\setminus N_G[v]$ is independent (by Claim 1), every edge of $G$ belongs to exactly one of
the following three categories: the $\Delta$ edges incident with $v$, the
$h$ edges having both end-vertices in $N_G(v)$, and the $e$ edges
joining $N_G(v)$ to $V(G)\setminus N_G[v]$. Therefore,
$|E(G)|=\Delta+h+e$, which together with \eqref{eq-new-22} and  \eqref{eq-new-1} yields
$\Delta+h+e\le2(\Delta+r+1)-3,$
which is equivalent to
\begin{equation}\label{eq:edge-count}
h+e\le\Delta+2r-1.
\end{equation}
Because $V(G)\setminus N_G[v]$ is independent (by Claim 1), the number $e$ is also
the sum of the degrees of the $r$ vertices outside $N_G[v]$. Each of
these degrees is at least $3$ by \eqref{eq:outside-degree}. Hence,
$e\ge3r$, which together with \eqref{eq:edge-count} gives
\begin{equation}\label{eq-new-03}
h\le\Delta-r-1.
\end{equation}
Every vertex outside $N_G[v]$ but different from $x_1,\ldots,x_q$ has
degree at least $3$ (because of  \eqref{eq:outside-degree}). Moreover, the $p-1$ common neighbors of $v$ and
$w$ give $p-1$ distinct edges with both end-vertices in $N_G(v)$, namely the edges joining
$w$ to these common neighbors. Thus, $h\ge p-1$, and hence, \eqref{eq:edge-count} yields
\begin{equation}\label{eq:edge-count-new}
e\le\Delta-p+2r.
\end{equation}
Since $e=\sum_{\alpha\in V(G)\setminus N_G[v]}d_{\alpha}(G)\ge\sum_{i=1}^{q}d_{x_i}(G)+3(r-q)$ and $q\le r$ (because $x_1,\ldots,x_q$ are among the $r$ vertices outside
$N_G[v]$), using \eqref{eq:edge-count-new}, we have
\begin{align}
\sum_{i=1}^{q}d_{x_i}(G)
&\le e-3(r-q)\le\Delta-r-p+3q\le\Delta-p+2q.
\label{eq:sum-x-upper}
\end{align}

On the other hand, for every $i\in\{1,\dots,q\}$, \eqref{eq:degree-comparison} gives
$d_{x_i}(G)\ge d_w(G)+1$, which together with
\eqref{eq:sum-x-upper} yields
\begin{equation}\label{eq:delta-p-lower}
q(d_w(G)+1)\le\sum_{i=1}^{q}d_{x_i}(G)
\le\Delta-r-p+3q,
\qquad \Delta-p\ge q(d_w(G)-1).
\end{equation}

We next estimate the sum of the degrees over
$N_G(v)\setminus N_G[w]$. In the sum of the degrees of the vertices in
$N_G(v)$, each of the $\Delta$ edges joining $v$ to $N_G(v)$ is counted
once, each of the $h$ edges within $N_G(v)$ is counted twice, and each
of the $e$ edges joining $N_G(v)$ to the vertices outside $N_G[v]$ is
counted once. Hence, this degree sum is $\Delta+2h+e$.

The set $N_G(v)\setminus N_G[w]$ is obtained from $N_G(v)$ by deleting
$w$ and the $p-1$ common neighbors of $v$ and $w$. The degree of $w$ is
$p+q$, while each common neighbor has degree at least $2$, because
it is adjacent to both $v$ and $w$. It follows that
\begin{align}
\sum_{z\in N_G(v)\setminus N_G[w]}d_z(G)
&\le\Delta+2h+e-p-q-2(p-1)\le3(\Delta-p)+r-q,
\label{eq:sum-z-first}
\end{align}
where the last inequality follows from \eqref{eq:edge-count} and \eqref{eq-new-03},

The first two inequalities in \eqref{eq:delta-p-lower} also give
$r\le\Delta-p-q(d_w(G)-2)$, which together with
\eqref{eq:sum-z-first} yields
\begin{equation}\label{eq:sum-z-upper}
\sum_{z\in N_G(v)\setminus N_G[w]}d_z(G)
\le4(\Delta-p)-q(d_w(G)-1)<4(\Delta-p).
\end{equation}

There are $\Delta-p$ vertices in $N_G(v)\setminus N_G[w]$: from the
$\Delta$ vertices in $N_G(v)$, one deletes $w$ and its $p-1$ common
neighbors with $v$. The Cauchy--Schwarz inequality gives
\[
\sum_{z\in N_G(v)\setminus N_G[w]}\frac1{d_z(G)}
\ge
\frac{(\Delta-p)^2}
{\displaystyle\sum_{z\in N_G(v)\setminus N_G[w]}d_z(G)},
\]
which together with \eqref{eq:sum-z-upper} gives
\begin{equation}\label{eq-new004}
\sum_{z\in N_G(v)\setminus N_G[w]}\frac1{d_z(G)}
>\frac{\Delta-p}{4}.
\end{equation}
%Also, \eqref{eq:sum-z-upper} gives
%\[
%-\sum_{z\in N_G(v)\setminus N_G[w]}
%\frac{d_z(G)}{\Delta(\Delta+q)}
%>-\frac{4(\Delta-p)}{\Delta(\Delta+q)}.
%\]
From \eqref{eq:sum-z-upper}
and \eqref{eq-new004}, it follows that   \begin{align}
q\sum_{z\in N_G(v)\setminus N_G[w]}
\left(\frac1{d_z(G)}-\frac{d_z(G)}{\Delta(\Delta+q)}\right)
\ge(\Delta-p)
\left(\frac q4-\frac{4q}{\Delta(\Delta+q)}\right).
\label{eq:first-sum-bound}
\end{align}
Also, we note that
\begin{align*}
(\Delta-p)\sum_{i=1}^{q}
\left(\frac1{d_{x_i}(G)}-
\frac{d_{x_i}(G)}{(p+q)(\Delta+q)}\right)>(p-\Delta)\sum_{i=1}^{q}
\frac{d_{x_i}(G)}{(p+q)(\Delta+q)},
\end{align*}
which together with \eqref{eq:sum-x-upper} yields
\begin{align}
(\Delta-p)\sum_{i=1}^{q}
\left(\frac1{d_{x_i}(G)}-
\frac{d_{x_i}(G)}{(p+q)(\Delta+q)}\right)
>{}-\frac{(\Delta-p)(\Delta-p+2q)}{(p+q)(\Delta+q)}.
\label{eq:second-sum-bound}
\end{align}

Using \eqref{eq:first-sum-bound} and \eqref{eq:second-sum-bound} in
\eqref{eq:main-difference}, we obtain \begin{align}
\SDD(G)-\SDD(G')
<{}&(p-\Delta)\left\{
\frac q4-\frac{4q}{\Delta(\Delta+q)}
-\frac{\Delta-p+2q}{(\Delta+q)(p+q)}
+\frac{q(\Delta+p+q)^2}{p\Delta (p+q)(\Delta+q)}
\right\}.
\label{eq:final-bound}
\end{align}
We denote by  $\Theta(p,q,\Delta)$ the expression within braces ``\{...\}'' in
\eqref{eq:final-bound}.
Then,
\begin{align}\label{eq-finall}
4\Delta p(p+q)(\Delta+q)\,\Theta(p,q,\Delta)
={}&\bigl[pq(p+q)-4(p-q)\bigr](\Delta-p-q)^2\nonumber\\
&+\bigl[2p^3q+p^2(5q^2-4)+3pq^3+16q^2\bigr]
(\Delta-p-q)\nonumber\\
&+q(p+q)\bigl[p(p+q)(p+2q)-12p+16q\bigr].
\end{align}
We recall that $\min(p,q)\ge1$ and $d_w(G)=p+q\le\Delta$. Since
\[
p(p+q)(p+2q)-12p+16q
\ge p(p+1)(p+2)-12p+16=p\bigl[(p+1)(p+2)-12\bigr]+16>0,
\]
$pq(p+q)-4(p-q)
\ge p(p+1)-4(p-1)
=p^2-3p+4>0,$ and $2p^3q+p^2(5q^2-4)+3pq^3+16q^2>0$,
from \eqref{eq-finall}, we have $\Theta(p,q,\Delta)>0$, and hence, \eqref{eq:final-bound} yields $\SDD(G)-\SDD(G')<0$, a contradiction.
\end{proof}

\end{document}